\documentclass[11pt]{amsart}
\usepackage{mathrsfs}
\usepackage{amsfonts}
\usepackage{latexsym,amsmath,amssymb}

\usepackage{xcolor} 
\usepackage{ulem}   
\usepackage{soul}   
\usepackage{array,tabularx}

\renewcommand{\epsilon}{\varepsilon}

\newtheorem{theorem}{Theorem}[section]

\newtheorem{lemma}[theorem]{Lemma}
 \newtheorem{prop}[theorem]{Proposition}

\newtheorem{deff}[theorem]{Definition}

\newcommand{\bth}{\begin{theorem}}
\newcommand{\ble}{\begin{lemma}}
\newcommand{\bcor}{\begin{corr}}
\newcommand{\bdeff}{\begin{deff}}
	\newcommand{\bprop}{\begin{proposition}}
\newcommand{\ele}{\end{lemma}}
\newcommand{\ecor}{\end{corr}}
	\newcommand{\edeff}{\end{deff}}
	
	\newcommand{\eprop}{\end{proposition}}

\newcommand{\la}{\lambda}

\renewcommand{\Pi}{\varPi}

\renewcommand{\epsilon}{\varepsilon}

\newcommand{\parital}{\partial}

\newcommand{\p}{\partial}

\numberwithin{equation}{section}

\thanks{The first author is supported  by NSFC Grant, No.12101145 and  Guangxi Science, Technology Project, Grant No. GuikeAD22035202. The second and last authors are supported by NSFC Grant No. 12431008. The third author is supported by NSFC Grant No.12141105.}

	\title
[ Hessian curvature flow with prescribed Gauss image ]{Second boundary value problem for the Hessian curvature flow}
\author{Rongli Huang}
\address{School of Mathematics and Statistics, Guangxi Normal University, Guilin, China}
\email{ronglihuangmath@gxnu.edu.cn}

\author{Changzheng Qu}
\address{School of Mathematics and Statistics, Ningbo University, Ningbo, China}
\curraddr{}
\email{quchangzheng@nbu.edu.cn}
\thanks{}

\author{ZhiZhang Wang}
\address{School of Mathematical Science, Fudan University, Shanghai, China}
\email{zzwang@fudan.edu.cn}

\author{Weifeng Wo}
\address{School of Mathematics and Statistics, Ningbo University, Ningbo, China}
\curraddr{}
\email{woweifeng@nbu.edu.cn}
\thanks{}

\begin{document}
\maketitle

\begin{abstract}
We investigate the evolution of strictly convex hypersurfaces driven by the $k$-Hessian curvature flow, subject to the second boundary condition.  We first explore the translating solutions corresponding to this boundary value problem. Next, we establish the long-time existence of the flow and prove that it converges to a translating solution. To overcome the difficulty of driving boundary $C^2$ estimates, we employ an orthogonal invariance technique. Using this method, we extend the results of  Schn\"urer-Smoczyk \cite{Schnurer2003} and Schn\"urer \cite{Schnurer2002} from the second boundary value problem of Gauss curvature flow to $k$-Hessian curvature flow.
\end{abstract}

\section{Introduction}
The evolution of convex hypersurfaces under curvature flows is a fundamental topic in differential geometry, with significant progress made over the past few decades. Prominent examples include flows driven by the mean curvature \cite{Angenent2019,EckerHuisken1991, Huisken1984} and the Gauss curvature \cite{Andrews2016, Andrews1999, Brendle2017, Firey, Tso1985}, which have yielded profound insights into the evolution and asymptotic behavior of hypersurfaces. For more results on curvature flows, see \cite{Andrews2020, Angenent2020, Choi2022, EckerHuisken1991} and references therein. However, the analysis becomes intricate when boundary conditions are involved. Various boundary value problems for curvature flows, including Dirichlet and Neumann conditions, have been explored \cite{Alschuler1994, Huisken1989, Schnurer20022}. Second boundary condition, corresponding to prescribed gradient images, naturally arises in optimal transport \cite{Chen2016, Chen2021, Kitagawa2012} and has been studied in relation to Hessian equations \cite{Jiang2018, Trudinger, Urbas1997, Urbas2001} as well as curvature equations \cite{Brendle2010, Urbas2002}. Nonlinear parabolic equations with second boundary conditions have also been addressed \cite{Huang2015, Wang2023}.

Schn\"urer and Smoczyk \cite{Schnurer2003} studied the second boundary value problem for Hessian flow equation:
\begin{equation}\label{SSflow}
        u_t=\log F(D^2u)-\log g(x,u,Du), \quad \text{in}\ \Omega\times[0,T),\\
\end{equation}
subject to second boundary condition $  Du(\Omega)=\Omega^*$ and initial data $  u|_{t=0}=u_0$. Here $F$ is a Hessian function of a special class, and $g(x,u,Du)$ is a positive function with $g_z>0$. It yields the logarithmic Gauss curvature flow:
\[u_t=\log \det u_{ij}-\frac{n+2}{2}\log(1+|Du|^2)-\log g_1(x,u).\]
which describes the evolution of a hypersurface as graph $u$, with the "vertical" velocity equals to the difference of the logarithms of the actual Gauss curvature for $u$ and the prescribed Gauss curvature $g_1(x,u)$.
  They established the existence of smooth, strictly convex solutions that converge to a solution $u_{\infty}$ of the elliptic second boundary value problem
\begin{equation}\label{limite}
    \begin{cases}
        F(D^2u_{\infty})=g(x,u_{\infty},Du_{\infty}), \quad \text{in} \ \ \Omega,\\
        Du(\Omega)=\Omega^*.
    \end{cases}
\end{equation}
Here,  \eqref{limite} includes the prescribed Gauss curvature equation:
\[\frac{\det u_{ij}}{(1+|Du|^2)^{\frac{n+2}{2}}} =g_1(x,u).\] 
Schn\"urer \cite{Schnurer2002} further considered the second boundary value problem for curvature flow:
\begin{equation}\label{Sflow}
    \begin{cases}
        X_t=-(\log F-\log f)\nu \quad &\text{in}\ \Omega\times[0,T),\\
      \nu(M(t))=\nu(M(0)),&\\
        X|_{t=0}=X_0, \quad  &\text{in}\ \Omega,
    \end{cases}
\end{equation}
where $X=(x,u(x,t))$ is a family of convex graphic hypersurfaces defined in bounded domain $\Omega$, and $M(t)=\text{graph}(-u)|_\Omega$. Since the unit normal $\nu$  acts as the Gauss map, $\nu(M(t))=\nu(M(0))$ is equivalent to a second boundary value condition $ Du(\Omega)=Du_0(\Omega)=\Omega^*$. Schn\"urer proved that the solution converge to translating solutions. In this flow, the hypersurface $X$ evolves such that its normal velocity is determined by the logarithmic difference of a curvature function $F$, from a special class, and a positive function $f$, where $F$ includes, in particular, the Gauss curvature. In related work, Huang \cite{Huang2015} investigated the second boundary problem for Lagrangian mean curvature flow, establishing flow convergence.

This paper aims to investigate the second boundary value problem for a more general curvature flow: the $k$-Hessian curvature flow, extending the results of \cite{Schnurer2002,Schnurer2003} to a broader class of curvature flows.  While the entire Hessian curvature flows in Minkowski space have been studied in \cite{WZ1,WZ2,WZ3}, these flows are not under second boundary conditions.
We consider the second boundary value problem for the evolution of strictly convex hypersurfaces  \(M_u(x,t)\) with normal velocity given by the $k$ Hessian curvature 
\begin{equation}\label{curvatureeq}
    \frac{\partial X}{\partial t} = \sigma_k^{\frac{1}{k}}(\kappa[M_u]) N, \quad  \text{in} \ \Omega,
\end{equation}
where $1 \leq k \leq n$,  \(N\) is the upward unit normal vector of $X(x,t)$. 
The convex graphic hypersurfaces \(M_u(x,t)\) in \(\mathbb{R}^{n+1}\) is described by  the embeddings $X:\Omega\times [0,T)\rightarrow \mathbb{R}^{n+1}$ as
\[
M_u(x,t) = \left\{ X(x,t) = (x, u(x, t)): \, (x,t) \in \Omega \times [0,T) \right\}.
\]
Here, $\Omega \subset \mathbb{R}^n$ is a strictly convex bounded domain with smooth boundary $\partial \Omega$ and $\kappa[M_u]=(\kappa_{1},\kappa_{2},\cdots,\kappa_n)$ denotes the principal curvatures of $M_u$. The Hessian curvature $\sigma_{k}$ is defined by the $k$-th  elementary symmetric polynomial in the principal curvatures:
\begin{equation*}
	\sigma_{k}(\kappa)=\sum_{1\leq i_{1}< i_{2}<\cdots< i_{k}\leq n}\kappa_{i_{1}}\kappa_{i_{2}}\cdots\kappa_{i_{k}}.
\end{equation*}

Thus, the graph function $u(x,t)$ of convex hypersurface $X(x,t)$, evolving by \eqref{curvatureeq}, satisfies the following evolution equation:
\begin{equation}\label{ueq}
\dot{u} = \sqrt{1 + |Du|^2} \sigma_k ^{\frac{1}{k}}(\kappa[M_u]),
\end{equation}
where \(\dot{u}=u_t\) denotes the time derivative of \(u\)  and \(Du\) represents the gradient with respect to the spatial variable $x$.

We consider the evolution starting from the initial hypersurface $M_u(x,0)$, given by
\[X(x,0)=X_0(x),
\]
and subject to the  second boundary condition
\begin{equation}\label{bc}
 Du(\Omega)=\Omega^*.
\end{equation}
Here $\Omega$ and $\Omega^*$ are smooth unifromly convex domains in $\mathbb{R}^n$. 

Given  $\Omega \subset \mathbb{R}^n$ and $X_0$, we study the following curvature flow problem:
\begin{equation}\label{Xflow}
	\left\{
	\begin{aligned}
		&X_t= \sigma_k^{\frac{1}{k}}(\kappa[M_u])N,  &\text{in } \Omega \times [0, T), \\
		&	 Du(\Omega)=\Omega^*,& \\
		&X(x,0)=X_0(x),  &\text{in } \Omega.
	\end{aligned}
	\right.
\end{equation}
The curvature flow problem \eqref{Xflow} can be reformulated as the following second boundary value problem for the graph function $u(x,t)$:
\begin{equation}\label{uflow}
	\left\{
	\begin{aligned}
		&	u_t = \sqrt{1 + |Du|^2} \sigma_k^{\frac{1}{k}} (\kappa[M_u]),  &\text{in } \Omega \times [0, T), \\
	&	Du (\Omega) = Du_0(\Omega)=\Omega^*,& \\
	&	u(x,0) = u_0(x),  &\text{in } \Omega.
	\end{aligned}
	\right.
\end{equation}

Our goal is to prove the existence of a smooth, strictly convex solution $u(x,t)$ for all $t \in [0, \infty)$. Moreover, as $t \to \infty$, the solution converges to a translating solution $u(x,t)=u_{\infty}(x)+at$ of \eqref{curvatureeq}, where $u(x)$ satisfies
\begin{equation}\label{limiteq}
	\left\{
	\begin{aligned}
		&\sigma_k^{\frac{1}{k}} (\kappa[M_{u_{\infty}}])=\frac{a}{\sqrt{1 + |Du_{\infty}|^2}},  &\text{in } \Omega,\\
		&	Du (\Omega) = \Omega^*.&
	\end{aligned}
	\right.
\end{equation}

We summarize our main result in the following theorem:
\begin{theorem}\label{t1.1}
	Let $\Omega$ and $\Omega^*$ be bounded strictly convex domains with smooth boundaries in $\mathbb{R}^n$. Then, there exists a unique smooth, strictly convex graphical solution $u(x,t)$ to \eqref{uflow} for all $t\geq 0$. Moreover,  as $t \to \infty$, $u$ converges to a translating solution with a specified velocity.
\end{theorem}

The $k$-Hessian curvature flow \eqref{curvatureeq} in Theorem \ref{t1.1}, including the mean curvature flow ($k=1$) and the Gauss curvature flow ($k=n$),  generalizes the results of Schn\"urer's flow\eqref{Sflow} to the broader class of $k$-Hessian curvatures $1\leq k\leq n$.  In \cite{HQWW}, we studied $k$-Hessian curvature hypersurfaces with prescribed Gauss images, using a  boundary $C^2$ estimate technique based on orthogonal invariance to extend Urbas’ work \cite{Urbas2002} from the Gauss curvature equation to the $k$-Hessian curvature equation. In this paper, we apply an orthogonal invariance method to obtain  boundary $C^2$ estimates for the flow. A key technique, using the maximum principle on the initial surface, enables the derivation of a double tangential boundary derivative estimate.  By adapting these boundary estimation techniques, we extend the study of $k$-Hessian curvature equation \cite{HQWW} to the corresponding $k$ Hessian curvature flow.

The paper is organized as follows. In Section 2, we describe some geometric notations, define the support function and the Legendre transformation, derive the corresponding equations, and state other relevant concepts and lemmas that will be used in the subsequent sections, respectively. In Section, we discuss the translating solution for the flow. Section 4 and 5 provide the $C^0$ and $w^*(F^*)^{-1}$ estimates. Section 6 addresses the strict obliqueness of the boundary condition. In Sections 7, we establish the $C^2$ estimates, including global and boundary $C^2$ estimates, where a special tangential direction is constructed in double tangential estimate. Finally, in Section 8, we complete the proof of the main theorem, establishing the long-time existence and convergence of the flow.

\section{Preliminaries}
\subsection{Geometric notations} 
We begin by introducing the notations used in this paper, which mainly follows the conventions of \cite{HQWW} and \cite{WZ2}. Unless otherwise stated, we employ the Einstein summation convention with repeated Latin indices summed from $1$ to $n$. The letter $C$ denotes a positive constant that may vary in different contexts but remains independent of the essential parameters in our analysis.
We use both $u_t$ and $\dot{u}$ to represent the time derivative of $u$, i.e. $$u_t = \dot{u} = \frac{\partial u}{\partial t}.$$
We denote the spatial derivatives of $u$ as
 $$u_i=D_{i}u=\dfrac{\partial u}{\partial x_{i}},\ u_{ij}=D_{ij}u=\dfrac{\partial^{2}u}{\partial x_{i}\partial x_{j}},\ |Du|=\sqrt{\sum_{i=1}^{n}|D_{i}u|^{2}}. $$
The upward-pointing unit normal vector to \(X\) is given by
\[
N(X) = \frac{(-Du, 1)}{\sqrt{1 + |Du|^2}}.
\]
In the standard coordinates of \(\mathbb{R}^{n+1}\), the induced metric $g_{ij}$ and its inverse  \(g^{ij}\)  on \(M_u\) are
\[
g_{ij} = \delta_{ij} + D_i u D_j u
\]
and
\[
g^{ij} = \delta_{ij} - \frac{D_i u D_j u}{1 + |Du|^2}.
\]
where $1\leq i,j\leq n$.
The second fundamental form $h_{ij}$ of \(M_u\) is
\[
h_{ij} = \frac{D_{ij} u}{\sqrt{1 + |Du|^2}}.
\]

The principal curvatures $\kappa$ of \(M_u\) are the eigenvalues of the Weingarten matrix \((h^j_i)=( h_{ik} g^{kj})\). According to \cite{Caffarelli1986}, these eigenvalues are the eigenvalues of the symmetric matrix
\[
A=(a_{ij}) = \left(\frac{1}{w} b^{ik}(D_{kl} u) b^{lj}\right),
\]
where \(w = \sqrt{1 + |Du|^2}\) and \[b^{ij} = \delta_{ij} - \frac{D_i u D_j u}{w(1 + w)}\] is the square root of \(g^{ij}\). Let \(b_{ij}\) be the inverse of \(b^{ij}\), given by
\[
b_{ij} = \delta_{ij} + \frac{D_i u D_j u}{1 + w}.
\]
Let \(\mathcal{S}\) denote the space of symmetric \(n \times n\) matrices, and define
\[
\mathcal{S}_+ = \{ A \in \mathcal{S} : \lambda(A) \in \Gamma_n \},
\]
where \(\Gamma_n = \{\lambda \in \mathbb{R}^n : \lambda_i > 0 \text{ for all } i\}\) is the convex cone, and \(\lambda(A) = (\lambda_1, \dots, \lambda_n)\) represents the eigenvalues of \(A\). Define the function \(F\) by
\[
F(A) = F\left( \frac{1}{w} b^{ik} (D_{kl} u) b^{lj} \right)= \sigma_k^{\frac{1}{k}}(\lambda(A)), \quad A \in \mathcal{S}_+.
\]
Note that the function \(F\) is well-defined for strictly convex hypersurfaces.
Set
\[
F^{ij}(A)= \frac{\partial F}{\partial a_{ij}}(A), \quad F^{ij, kl} = \frac{\partial^2 F}{\partial a_{ij} \partial a_{kl}}.
\]
Thus, using above notations, evolution \(\eqref{ueq}\) can be rewritten as
\begin{equation}\label{ueq2}
u_t= \sqrt{1 + |Du|^2} F\left( \frac{1}{w} b^{ik} (D_{kl} u) b^{lj} \right)=wF. 	
\end{equation}

\subsection{Support function}
The support function \( v \) for the convex hypersurface \(X
\) is defined as
\[v:=-\langle X, N\rangle=\frac{1}{\sqrt{1+|Du|^2}}\left(\sum_ix_i\frac{\partial u}{\partial x_i}-u\right).\]
where $\langle \cdot, \cdot\rangle$ denotes the standard Euclidean inner product.

Let $\{\textbf{e}_1, \cdots, \textbf{e}_n\}$ be an orthonormal frame on $\mathbb{S}^n$ and let $\nabla$ be  the standard Levi-Civita connection of $\mathbb{S}^n$.
Then the spherical Hessian of $v$ is
\begin{equation}\label{sphere}
	\Lambda_{ij}=\nabla_{ij}v+v\delta_{ij},
\end{equation}where $\nabla_iv, \nabla_{ij}v$ are the first and second covariant derivatives with respect to the standard spherical  metric.
From convex geometry theory, we  have the following
$$\nabla_iv=-\langle X, \textbf{e}_i\rangle, \ \ X=-\sum_i(\nabla_iv)\textbf{e}_i-vN.$$
Moreover, we have
\begin{eqnarray*}
	g_{ij}=\sum_k\Lambda_{ik}\Lambda_{kj},\ \ h_{ij}=\Lambda_{ij}.
\end{eqnarray*}
This implies that the eigenvalues of the spherical Hessian are the curvature radii of $M_u$. Specifically,
if the principal curvatures of $M_u$ are $(\la_1, \cdots, \la_n),$ then the eigenvalues of the spherical Hessian
are $\left(\la_1^{-1}, \cdots, \la_n^{-1}\right).$
Therefore, flow equation \eqref{curvatureeq}  can be written as
\begin{equation}
	v_t=-\frac{1}{F^*(\Lambda)},\quad x\in \mathbb{S}^n_+,
\end{equation}
where $F^*(\Lambda)= \left( \frac{\sigma_n}{\sigma_{n-k}}(\lambda(\Lambda)) \right)^{\frac{1}{k}}$.

\subsection{Legendre transform}
The Legendre transform of $u$, denoted by $u^*:\Omega^*\times [0,T)\rightarrow \mathbb{R}$, is defined by
\[ u^*=\sum_{i=1}^n x_iy_i-u, \ \ y_i=\frac{\p u}{\p x_i}.\]
From the definition, we obtain the following properties:
\[\dot{u}^*=-\dot{u}, \quad\frac{\partial u^*}{\partial y_i}=x_i,\quad\frac{\partial^2u^*}{\partial y_i y_j}=u^{ij},\]
where $u^{ij}$ is the inverse of Hessian $u_{ij}$ and $y_i$ is considered time-independent.
From the theory of convex bodies and boundary condition \eqref{bc}, we get
\begin{equation}\label{boundary2}
	Du^* (\Omega^*) = \Omega.
\end{equation}
 It is also known that
$$\left(\frac{\p^2 u}{\p x_i\p x_j}\right)=\left(\frac{\p^2  u^*}{\p y_i\p y_j}\right)^{-1}.$$
Using the coordinate $(y_1,y_2,\cdots,y_n)$, the first and the second fundamental forms can be rewritten as:
$$g_{ij}=\delta_{ij}+y_iy_j, \ \ h_{ij}=\frac{u^{* ij}}{\sqrt{1+|y|^2}},$$
where $\left(u^{* ij}\right)$ denotes the inverse matrix of $(u^*_{ij})$ and $|y|^2=\sum_iy_i^2$. Let $W$ denote the Weingarten matrix of $M_u,$ then we have
$$(W^{-1})_{ij}=\sqrt{1+|y|^2}\sum_kg_{ik}u^*_{kj}.$$

Then, we obtain the evolution equation for $u^*$:
\[
(u^*)_t = -\sqrt{1 +|Du|^2} \frac{1}{F^* (\lambda^*)}= -w^* (F^*)^{-1},
\]
where $w^*=\sqrt{1+y^2}$, and
\begin{equation}\label{Legendre}
F^* (\lambda^*) =	F^*\Big(w^* \sum_{k,l}b^*_{ik} (D_{kl}u^*) b^*_{lj}\Big)= \left( \frac{\sigma_n}{\sigma_{n-k}} (\lambda^*) \right)^{\frac{1}{k}},
\end{equation}
where $\lambda_i^*=1/\kappa_i$.
Thus, we derive the following dual evolution problem for $u^*$:
\begin{equation}\label{dualeq}
	\begin{cases}
		(u^*)_t = -w^* (F^*)^{-1}, \  \text{in } \Omega^* \times [0, T), \\
		Du^* (\Omega^*) = \Omega, \\
		u^*(y,0)= u^*_0(y),\ \  \quad\text{in}\  \Omega^*,\ \
	\end{cases}
\end{equation}
where \(u^*_0 = x \cdot Du_0 - u_0\).

\subsection{Second boundary conditions}
For convenience of boundary estimate, we introduce a \textbf{defining function }$h$ for convex domain $\Omega^*$  to reformulate the boundary condition $Du(\Omega)=\Omega^*$ to
 \begin{equation}\label{bdh}
	h(Du)=0 \ \ \text{on}\ \ \partial \Omega.
\end{equation}
Here $h : \mathbb{R}^n \to \mathbb{R}$ is a smooth, strictly concave function such that
\[\Omega=\{p\in \mathbb{R}^n:\ h(p)>0\}\ \text{and}\ |Dh|\ne 0 \ \text{on} \ \partial \Omega^*.\]
Without loss of generality, we may assume
$|Dh|=1 \ \ \text{on}\  \partial \Omega^*.$ Note that $h(Du)$ is positive in $\Omega$ and zero on $\partial \Omega$.

 Similarly, we define the defining function $h^*$ for $\Omega$  and have
  \begin{equation}\label{bdh2}
 	h^*(Du^*)=0 \ \ \text{on}\ \ \partial \Omega^*.
 \end{equation}
 \subsection{Previous results}
 We recall some basic results and lemmas from previous work \cite{HQWW} that will be used in our $C^2$ estimates.
 \begin{lemma}\label{lemma1}{(Lemma 5.1 of \cite{HQWW})} 	
 Let $(y_1,\cdots,y_n)$ be the rectangular coordinate of $\mathbb{R}^n$. Suppose $\nabla$ is the standard Levi-Civita connection on the unit sphere. Let
 $
 	e_i = \sum_kw^* \, b_{ik}^* \partial/\partial y_k.
 $
 	Then $\{e_1,\cdots,e_n\}$ is an orthonormal frame on $\mathbb{S}^n$. Moreover, for any function $v$ defined on some subset of $\mathbb{R}^n$, we get
 	\begin{eqnarray}\label{form}
 		\sum_{k,l}w^* \, b_{ik}^* D_{kl}v \, b_{lj}^*=\nabla^2_{ij}\frac{v}{w^*}+\frac{v}{w^*}\delta_{ij}.
 	\end{eqnarray}
 \end{lemma}

 \begin{prop}\label{propT}{(Proposition 5.2 of \cite{HQWW})} 	
 	Given some boundary point $y_0\in\p\Omega^*$ and some tangential vector $\xi=(\xi_1,\cdots,\xi_n)$ of $\p\Omega^*$ at $y_0$. We can construct a vector field $\mathcal{T}=\sum_m\mathcal{T}_m\frac{\p}{\p y_m}$ near $y_0$. It satisfies
 	\begin{equation}\label{5.2}
 		\mathcal{T}_m(y_0)=\sqrt{1+|y_0|^2}\xi_m
 	\end{equation}
 	and
 	\begin{equation}\label{Tineq}
 		|\mathcal{T}(y)|^2=\sum_m\mathcal{T}_m^2(y)\leq 1+|y|^2 \text{ for any } y,
 	\end{equation}
 	where the equality of \eqref{Tineq} holds when $y=y_0$.
  \end{prop}
\section{Translating Solutions}

In this section, we establish the existence of a translating solution to the \(k\)-Hessian curvature flow \eqref{uflow}. A translating solution \( u(x,t) \) takes the form
\[
u(x,t) = u(x) + at,
\]
where \( a \) is a constant. Substituting this into the flow equation \eqref{ueq}, we obtain the equation for $u(x)$:
\begin{equation}\label{teq}
	\frac{a}{\sqrt{1+|D u|^2}} = \sigma^{\frac{1}{k}}(\kappa(M_u)).
\end{equation}
 By Theorem 1.3 in \cite{HQWW}, we know that there exists a  unique solution \( u(x) \)  to
 \begin{equation}
 	\left\{
 	\begin{aligned}
 		&\sigma_k^{\frac{1}{k}}(M_u) = a N\cdot E,  &\text{in } \Omega\\
 		&	Du (\Omega) = \Omega^*,&
 	\end{aligned}
 	\right.
 \end{equation}
where \( a \) is a constant determined by  \( \Omega \) and \( \Omega^* \), $
N(X) = \frac{(-Du, 1)}{\sqrt{1 + |Du|^2}}
$, and $E=(0,\cdots,0,1)$.
This shows the existence of a solution to \eqref{teq} subject to the boundary condition, thus establishing the existence of translating solution to \eqref{limiteq}.

Applying the Legendre transform to the function \( u(x,t) \), we have
\[
u^*(y,t) = x(t) \cdot D u(x,t) - u(x,t) = u^*(y) - at.
\]
This leads to the equation for \( u^* \):
\[
\frac{a}{\sqrt{1+|y|^2}} = (F^*)^{-1}
\]
or equivalently,
\[
F^* \left( w^* b^*_{ik} D_{kl} u^* b^*_{lj} \right) = \frac{\sqrt{1 + |y|^2}}{a}.
\]
Thus, we also have the existence of translating solution to the dual problem.

\section{\(C^0\) Estimates}

In this section, we first derive an estimate for $\dot{u}$, then obtain the \(C^0\) estimate for \( u(x,t) \) and $u^*(y,t)$. Let $\nu$ and $\nu^*$ denote the inner unit normal of $\Omega$ and $\Omega^*$, respectively.

We first have the following inequality for the $\sigma_k$ and $\sigma_1$ curvatures:
\[
\sigma_k^{1/k}(\kappa) = F\left(\frac{1}{w} b^{ik} (D_{kl} u) b^{lj}\right) \leq C \sigma_1 \left(\frac{1}{w} b^{ik} (D_{kl} u) b^{lj}\right) + C.
\]
Integrating this inequality over \( \Omega \)  yields:
$$\int_{\Omega} F(a_{ij})dx\leq \int_{\Omega}D\left(\frac{Du}{w}\right)dx+C|\Omega|=C|\Omega|-\int_{\p\Omega}\frac{D_{\nu}u}{w}\leq C, $$
where $a_{ij} = \frac{1}{w} b^{ik} (D_{kl} u) b^{lj}$ and $\nu$ is the unit normal to $\Omega$. Here, the boundedness of $\Omega$ and $\Omega^*$
ensures that the integral is bounded:
\begin{equation}\label{Fbound}
\int_{\Omega} F(a_{ij}) \, dx \leq C.
\end{equation}

By the boundedness of $\Omega$ and $\Omega^*$, we can derive the following $C^1$-estimate:
\begin{lemma}\label{C1}
Let $\Omega$ and $\Omega^*$ be bounded domains. For a smooth convex solution $u(x,t)$ of \eqref{uflow}, the gradient of
$u$ satisfies:
\[|Du|\leq C, \ \text{for all}\ t\geq 0,\]
where $C$ is a constant depending on $\Omega$ and $\Omega^*$.

\end{lemma}

Integrating both sides of the graph equation $u_t =\sqrt{1 + |Du|^2}F(a_{ij})$ over $\Omega$, and 
using \eqref{Fbound} and the $C^1$ bound, we obtain:
\[
\int_{\Omega} u_t dx\leq C.
\]
Next, integrating this inequality with respect to time from \( 0 \) to \( t \), we have
\[
\int_{\Omega} u(x,t)dx - \int_{\Omega} u(x,0)dx = \int_0^t \int_{\Omega} u_t \, dx \, d\tau\leq C t.
\]
Thus, we conclude that:
\[
\int_{\Omega} u(x,T)\leq C T+C .
\]
Therefore, there exists a point \( \tilde{x} \in \Omega \) such that \( u(\tilde{x}, T) \leq C(\Omega, T) \). By $|Du|\leq C$, we get
\[
u(x,T) \leq C(\Omega,T).
\]

 Since \( u_t > 0 \),  it follows that:

\[
u(x, t) - u(x, 0) > 0.
\]
Thus, we have
\[
|u(x,T)| \leq C(\Omega, T).
\] Moreover, we can also derive the bound for \( u^*=x \cdot Du - u \):
\[
|u^*| \leq C(\Omega, \Omega^*, T).
\]
We have established \( C^0 \) estimate for  $u$ and \( u^* \).

\section{ Estimate of \( w^* (F^*)^{-1}\)}

From the dual flow equation \eqref{dualeq}, we deduce the evolution of \( (F^*)^{-1} \):
\begin{align*}
		\frac{d}{dt}(F^*)^{-1}&= - (F^*)^{-2} F^{*ij}
(w^* b^*_{ik} D_{kl} \frac{d}{dt}(u^*) b^*_{lj} )\\
	&= (F^*)^{-2} (F^*)^{ij}
	\cdot\left[ w^* b^*_{ik} D_{kl} (w^* (F^*)^{-1}) b^*_{lj} \right] ,
\end{align*}
where $w^*=\sqrt{1+|y|^2}$ and $b^*_{ij}=\delta_{ij}+\frac{y_iy_j}{1+w^*}$.

This gives the evolution equation for \( w^* (F^*)^{-1} \):
\[
\frac{d}{dt} \left( w^* (F^*)^{-1} \right) = w^* (F^*)^{-2} (F^*)^{ij} \left[ w^* b^*_{ik} D_{kl} (w^* (F^*)^{-1}) b_{ij}^* \right].
\]
Thus, $ w^* (F^*)^{-1}=-\dot{u}^*  $ satisfies a parabolic equation. Applying the maximum principle, we can conclude that $ w^* (F^*)^{-1}$ attains its maximum or minimum on the parabolic boundary. By the Hopf maximum principle, at the maximum or minimum point of \( \dot{u}^* \), we have
\[
\dot{u}^*_{\nu^*}  <0,\quad (\text{or}>0) ,
\]
unless $\dot{u}^*$ is constant. 

On the other hand, by the boundary condition $h^*(Du^*) = 0$, we get
\begin{equation}
	h^*_{p_k} \dot{u}^*_k = 0\quad \text{on}\quad \partial\Omega.
\end{equation}
According to  Lemma 3.1 of \cite{Schnurer2002}, as long as the solution exists, the boundary condition is strictly oblique, i.e.
\begin{equation}
	\left<\nu(x),\nu^*(Du(x,t))\right> >0,\ \ x\in \partial\Omega,
\end{equation}
where $\nu$ and $\nu^*$ denote the inner unit normal of $\Omega$ and $\Omega^*$, respectively.
Moreover, by boundary condition $h(Du)=0$ on $\partial \Omega$ and $|Dh|=1$, we know that the interior unit normal of $\partial\Omega^*$ is given by
$$\nu^*(p)=Dh(p_1,\cdots,p_n)=\left(\frac{\p h}{\p p_1},\cdots, \frac{\p h}{\p p_n}\right).$$
Thus we get a contradiction by the strict obliqueness of of $h^*_{p_k}$.
Therefore, we conclude the following estimate for $\dot{u}^*=- w^* (F^*)^{-1}$:
\begin{prop}
	Let \( u^* \) be a smooth convex solution of the dual flow \eqref{dualeq}. As long as the smooth convex solution  \(u^* \)  exists, we always have
	\[
	\min \left\{ \min_{t=0}(-\dot{u}^*), 0\right\} \leq -\dot{u}^* \leq \max \left\{ \max_{t=0} (-\dot{u}^*), 0 \right\}.
	\]
\end{prop}

\section{Estimates of Strict Obliqueness }
In this section, we quantify the strict obliqueness of the boundary conditions. The proof integrates ideas from \cite{Urbas1997} and \cite{Urbas2001}. We adapt Lemma 4.2 in \cite{HQWW} from elliptic equations to the flow. A brief outline of the proof is provided below.

\begin{prop}
		Suppose $\Omega$, $\Omega^*$ are bounded strictly convex domains with smooth boundary in $\mathbb{R}^{n}$. For a smooth solution of \eqref{dualeq}, we have the strict obliqueness estimate:
	\begin{equation}\label{soe}
		\langle\nu^*(Du(x,t)), \nu(x)\rangle\geq c_0>0,\ x\in\partial \Omega,
	\end{equation}
	where $\nu$ and $\nu^*$ are the inner unit normal of $\Omega$ and $\Omega^*$ and $c_0$ is some constant independent of time.
\end{prop}
\begin{proof}
	 Let  $h(Du)$ be the defining function of $\Omega^*$  giving $\nu^* = Dh=(h_{p_1},\cdots,h_{p_n})$. Define
	 $$\chi(x,t)=\langle\nu^*(Du(x,t)),\nu(x)\rangle=h_{p_k}(Du(x,t))\nu_k(x),$$
	 where $h_{p_k}=\frac{\p h}{\p p_k}$ and $\nu=(\nu_1,\cdots,\nu_n)$.
According to \cite{Urbas2001},  we can get $\chi\geq 0$ and
	\begin{eqnarray}\label{a}
		\chi=\sqrt{u^{ij}\nu_i\nu_j(D_{kl}u) h_{p_k}h_{p_l}},  \ \ \text{ on }\ \  \p\Omega.
	\end{eqnarray}
To show that there exists a positive lower bound for $\chi$ for $(x,t)\in\partial\Omega\times [0,T]$, it suffices to prove $u_{kl} h_{p_k}h_{p_l}$ and $u^{ij}\nu_i\nu_j$ have uniform positive lower bounds.
	
	Suppose  $(x_0,t_0)\in\p\Omega\times [0,T]$ is the point where $\chi $ attains its minimum on $\p\Omega$. Define	$$\phi(x,t) = \chi(x,t) + A h(Du(x,t)) =  h_{p_k}(Du(x,t)) \nu_k(x,t) + A h(Du(x,t)),
	$$
	where $A$ is a positive sufficiently large constant to be determined.  Following the argument in \cite{Urbas2001}, to prove that $u_{kl} h_{p_k}h_{p_l}$  has a uniform positive lower bound, it is sufficient to show
	\begin{eqnarray}\label{phi1}
		\phi_n(x_0,t_0)\geq -C,
	\end{eqnarray}
	where $e_n$ is the unit outward normal of $\p\Omega$ at $x_0$.
	
Suppose $h^*(x)$ is the uniformly concave function of $\Omega$,  giving $\nu = Dh^*$. Define $$\chi^*(y,t)=\langle\nu^*(y),\nu(Du^*(y,t))\rangle, $$ then $\chi^*(y,t)=\chi(x,t)$, where $y=Du(x,t)$. Similarly, to prove the lower bound for  $u^{ij}\nu_i\nu_j$, define
$$\phi^*(y,t)=\chi^*(y,t)+A^* h^*(Du^*(y,t)),$$
it suffices to show:
\begin{eqnarray}\label{phi2}
	\phi^*_{n^*}(y_0,t_0)\geq -C,
\end{eqnarray}
where $(y_0,t_0)$ is the minimum point and $e_n^*$ is the unit outward normal of $\p\Omega^*$ at $y_0$.

Therefore, proving \eqref{phi1} and \eqref{phi2} yields the desired strict obliqueness estimate \eqref{soe}.

We first prove \eqref{phi1}. Denote $G(Du,D^2u)$ by:
$$G(Du,D^2u)=wF\left( \frac{1}{w} b^{ik} (D_{kl} u) b^{lj} \right)=w\sigma_k^{\frac{1}{k}}(a_{ij}).$$

For the evolution equation $u_t=G=w F$, we define the operator \( \mathcal{L}\) as:
\begin{equation}
	\mathcal{L} f = -\dot{f} +  \left( G^{ij} D_{ij} f +G^sD_s f \right),
	\end{equation}
for any function $f$, where
$$G^{ij}=\frac{\p G}{\p u_{ij}}=wF^{kl}\frac{\partial a_{kl}}{\partial u_{ij}}=\sum_{p,q}F^{pq}b^{ip}b^{qj},$$
and
\begin{align*}
		G^{s}=&\frac{\partial w}{\partial u_s}F+wF^{kl}\frac{\partial a_{kl}}{\partial u_{s}}=\frac{u_{s}}{w}F+wF^{ij}u_{kl} \frac{\partial}{\partial u_s}\left( \frac{1}{w} b^{ik}  b^{lj} \right)\\
=&\frac{u_{s}}{w}F-\frac{u_s}{w}F^{ij}a_{ij}-2wF^{ij}a_{ik}b^{lj}\frac{\partial b_{kl}}{\partial u_{s}}=-2wF^{ij}a_{ik}b^{lj}\frac{\partial b_{kl}}{\partial u_{s}}.
\end{align*}

Differentiating $\dot{u}=wF=G$ with respect to $x_l$ yields
\[\dot{u}_l = G^{ij} u_{ijl}+  G^{s} u_{sl}.\]
We also have\begin{equation*}
	\phi_s=h_{p_kp_l} u_{sl} \nu_{k}+ h_{p_k} \nu_{ks}+Ah_{p_k}u_{ks}
\end{equation*}
and
\begin{equation*}
	\phi_{ij}=u_{jm}(h_{p_{k}p_{l}p_{m}}\nu_{k}+A h_{p_{l}p_{m}})	+2G^{ij}h_{p_{k}p_{l}}u_{li}\nu_{kj}+G^{ij}h_{p_{k}}\nu_{kij}.
\end{equation*}

Thus we get
	\begin{align*}
	\mathcal{L}\phi=&-\dot{\phi} +  \left( G^{ij} D_{ij} \phi +G^sD_s \phi \right)\\
	=&G^{ij}u_{il}u_{jm}(h_{p_{k}p_{l}p_{m}}\nu_{k}+A h_{p_{l}p_{m}})\\
&\qquad\qquad 	+2G^{ij}h_{p_{k}p_{l}}u_{li}\nu_{kj}+G^{ij}h_{p_{k}}\nu_{kij}+G^{s}h_{p_{k}}\nu_{ks}\\
	\leq& (h_{p_{k}p_{l}p_{m}}\nu_{k}+A h_{p_{l}p_{m}}+\delta_{lm})G^{ij}u_{il}u_{jm} + C \sum_{i} G^{ii} + C \sum_s G^s.
\end{align*}

	
	In Section 5, we have proved an estimate for $w^*(F^*)^{-1} $. Specifically, the estimate \( w^*(F^*)^{-1}  = wF= G \) and bound of $|Du|$ provide an upper bound for $G$ and $\sum_s G^s$.  Since $D^{2}h\leq-\theta I$ for some positive constant $\theta$, we can choose sufficiently large $A$ such that:
	$$h_{p_{k}p_{l}p_{m}}\nu_{k}+A h_{p_{l}p_{m}}+\delta_{lm}<0.$$
Using the fact that \( \sum G^{ii} \geq C>0 \), we derive
\begin{equation}\label{phiineq1}
	\mathcal{L} \phi \leq C \sum_i G^{ii}.
\end{equation}

Next, we shall construct a barrier function and apply the maximum principle to establish the desired lower bound. The argument follows similarly to the proof of Lemma 4.2 in \cite{HQWW}, with the only modification of the auxiliary function.

Without loss of generality, we assume that $x_{0}$ is the origin and the positive $x_{n}$-axis corresponds to the interior normal direction to $\partial\Omega$
at the origin. Suppose near the origin, the boundary $\partial\Omega$ is given by
\begin{equation}\label{e3.12}
	x_{n}=\rho(x')=\frac{1}{2}\sum_{\alpha<n}\kappa^{b}_{\alpha}x^{2}_{\alpha}+O(|x'|^{3}),
\end{equation}
where $\kappa^{b}_{1},\kappa^{b}_{2},\cdots,\kappa^{b}_{n-1}$ are the principal curvatures of $\partial\Omega$  at the origin and $x'=(x_{1},x_{2},\cdots,$ $x_{n-1})$.
Denote a neighborhood of $x_{0}$ in $\Omega$ by
$$\Omega_{r}:=\{x\in\Omega:\rho(x')< x_{n}<\rho(x')+r^{2}, |x'|<r\},$$
where $r$ is a small positive constant to be chosen.

Define a barrier function as follows:
\begin{equation}\label{boundary}
	\varphi(x)=-\rho(x')+ x_{n}+\delta|x'|^{2}-Kx^{2}_{n},
\end{equation}
where $\delta=\frac{1}{6}\min\{\kappa^{b}_{1},\kappa^{b}_{2},\cdots,\kappa^{b}_{n-1}\}$ and $K$ is some large undetermined positive constant. Note that $-\varphi$ is strictly convex when $r$ is sufficiently small. Similar to Lemma 4.2 of \cite{HQWW}, we have
\begin{equation}\label{phiineq1}
	\mathcal{L}\varphi
	\leq -C\sum G^{ii},\,\,x\in\Omega_{r},
\end{equation}
by choosing $K$  sufficiently large and $r$  sufficiently small. Moreover, we have
\[\varphi\geq 0,  \quad\text{on}\quad \partial \Omega_{r}.\]
Now we consider the following auxiliary function
\[
\Phi(x,t) = \phi(x,t) - \phi(x_0, t_0) + B \varphi(x),
\]
where \( B \) is a constant to be determined.

By \eqref{phiineq1}, we get 
\begin{equation*}
	\mathcal{L}(\Phi(x))\leq C(1-B)\sum_i G^{ii}<0
\end{equation*}
for large $B$.

By letting $K$ being large enough we conclude that
\begin{equation}\label{e3.15aa}
	\left\{ \begin{aligned}
		\triangle(\Phi(x,0))&\leq 0,\ \  &&x\in\Omega_{r},\\
		\Phi(x,0))&\geq 0,\ \  &&x\in\partial\Omega_{r}.
	\end{aligned} \right.
\end{equation}
By the maximum principle, we get
\begin{equation}\label{e3.15ab}
	\Phi(x,0)|_{\Omega_{r}}\geq 0.
\end{equation}
Then we can derive
\begin{equation}\label{e3.15aaaa}
	\Phi(x,t)|_{\Omega_{r}\times[0,T)}\geq 0.
\end{equation}
Combining  with $\Phi(x_0,t_0)=0$, we have $\partial_n\Phi(x_0)\geq 0$, which proves \eqref{phi1}.

The proof for the second inequality $\eqref{phi2}$ follows similarly. Here we use the linear operator \( \mathcal{L}^* \):
\[ \mathcal{L}^* f = - \dot{f} + (w^*)^{2} (F^*)^{-2} (F^*)^{ij} b^*_{ik} (D_{kl} f) b^*_{lj},
\] for dual equation $u^*_t=-w^*(F^*)^{-1}$. We refer to \cite{HQWW} for rest of the proof.

Thus, we derived the strict obliqueness estimate \eqref{soe}.

\end{proof}

\section{$C^2$ estimates}

In this section, we derive $C^2$ a priori for $u^*$. We assume the  existence of flow \eqref{dualeq} over the time interval
$[0,T]$. Notably, the estimate will not depend on $T$.

\subsection{Global $C^2$ estimates}

Let  \( v = -X\cdot N \) be the support function  of hypersurface $X(x,t)$,
where $N$ is the unit normal. 
Suppose $\{\textbf{e}_1,\cdots, \textbf{e}_n\}$ is an orthonormal frame on $\mathbb{S}^n_+$. The spherical Hessian is given by:
\[
\Lambda_{ij} = \nabla_{ij} v + v \delta_{ij},
\]
and satisfies the following properties:
\[
\nabla_k\Lambda_{ij} =\nabla_j \Lambda_{ik}
\]
and
\[
\nabla_{jj}\Lambda_{ii} -\nabla_{ii} \Lambda_{jj} = -\Lambda_{jj} + \Lambda_{ii}.
\]
 Let $\tilde{\Omega} = P^{-1}(\Omega^*)$. Then we get the evolution equation for $v$
\begin{equation}
	\left\{
	\begin{aligned}
		&v_t = - (F^*)^{-1} ( \Lambda_{ij} ), \ x\in\tilde{\Omega},\\
		&v(\cdot, 0)=u_0^*x_{n+1}.
	\end{aligned}
	\right.
\end{equation}

Consider the problem
\[
\sup_{\bar{\tilde{\Omega}}\times[0,T]\atop \eta \in T_y(\partial\Omega^*), |\eta|=1} \Lambda_{\eta\eta}(x).
\]
Without loss of generality, we can assume at $(x_0,t_0)$, $\textbf{e}_1$ are the maximum point and direction of the above problem. Moreover, we assume
$\Lambda_{ij}$ is a diagonal matrix at $x_0$.
Consequently, we have
\[
\Lambda_{\eta \eta} \leq \Lambda_{11}(x_0,t_0).
\]

Now, consider the test function
$$\varphi = \Lambda_{11}.$$
We have
\[
\varphi_t = ( \Lambda_{11})_t=(v_{11} + v)_t = (v_t)_{11} + v_t.
\]
The first and  second derivative of $v_t$ are calculated as follows:
\[
(v_t)_1 = (F^*)^{-2}  F^{*ij} \Lambda_{ij1}
\]
and
\begin{align*}
(v_t)_{11}&= (F^*)^{-2} \left( F^* \right)^{ij} \Lambda_{ij11}
+ (F^*)^{-2} (F^*)^{ij,pq} \Lambda_{ij1} \Lambda_{pq1}
- 2 (F^*)^{-3} ((F^*)^{ij} \Lambda_{ij1} )^2\\
&\leq (F^*)^{-2} \left( F^{*ii}\Lambda_{11ii} -  \sum_i F^{*ii} \Lambda_{11} +  F^{*ii} \Lambda_{ii} \right)
- 2 (F^*)^{-3} ( F^{*ii} \Lambda_{ii1})^2.
\end{align*}
Thus, we obtain
\begin{align*}
	\varphi_t &\leq (F^*)^{-2} ( F^{*ii} \varphi_{ii} - \varphi \sum_i F^{*ii} + F^*)\\
&	- 2 (F^*)^{-3} (F^{*ii}  \Lambda_{ii1})^2 - (F^*)^{-1}.
\end{align*}
Therefore,
\begin{equation}\label{phiineq}
	\varphi_t - (F^*)^{-2} F^{*ii} \varphi_{ii}  \leq - \varphi (F^*)^{-2} \sum F^{*ii} .
	\end{equation}
	
If the maximum value of \(\varphi\) is attained at an interior point \((x_0, t_0) \in \tilde{\Omega} \times (0, T]\), then from \eqref{phiineq} implies a contradiction. We get the maximum is achieved on boundary. Hence, we conclude that:
\[
\sup_{\bar{\tilde{\Omega}}\times [0, T]} \Lambda_{\eta\eta} \leq C + \sup_{\partial \tilde{\Omega}\times (0, T_1] \cup \tilde{\Omega} \times \{0\}}\Lambda_{\eta\eta},
\]
Therefore, we obtain
\[
\sup_{\bar{\Omega}^*\times [0, T]} \sum_{k,l}w^* b_{ik}^* D_{kl}^2 u^* b_{jl}^* \leq C + \sup_{\partial \bar{\Omega}^*\times [0, T]} \sum_{k,l}w^* b_{ik}^* D_{kl}^2 u^* b_{lj}^*,
\]
which implies
\begin{equation}\label{supest1}
\sup_{\bar{\Omega}^* \times [0, T]} |D^2 u^*| \leq C \left( 1 + \sup_{\partial \Omega^* \times [0, T]} |D^2 u^*| \right).
\end{equation}

\subsection{Boundary $C^2$ estimates}
We now proceed to estimate the second derivatives of $u^*$ on $\partial\Omega^*\times(0,T]$.

For the dual problem, the boundary condition is given by:
\begin{equation}\label{bc2}
	h^*(Du^*)= 0,\quad x\in \partial\Omega^*,
\end{equation}
where $h^*$ is the defining function of $\Omega$. Let $\beta=(h^*_{p_1}, \ldots, h^*_{p_n})$. Denote $\nu$ as the interior unit normal of $\p\Omega$. Thus, we have $\beta(y)=\nu(Du^*(y))$. By the strictly obliqueness of the boundary condition, it follows that  $\beta \cdot \nu^* > 0$, where $\nu^*$ is the interior unit normal of $\p\Omega^*$.

Consider a fixed point on $\Omega^* \times [0, T]$, any direction at this point can be decomposed as
\[a\tau+b\beta,\]
where $\tau$ is a tangent direction to $\partial \Omega^*$. Here the constants $a$ and $b$ are uniformly bounded due to the strict obliqueness estimates.  Therefore, it suffices to estimate the following second derivatives:
\[u^*_{\tau\beta},\ u^*_{\beta\beta},\ u^*_{\tau\tau}. \]

\noindent $\bullet$ \textbf{Estimate for $u^*_{\tau\beta}$:}

Differentiating the boundary condition $h^*(Du^*)= 0,\quad x\in \partial\Omega^*$ along the tangential direction $\tau$, we have

\begin{equation}\label{estimatetau}
	u^*_{\tau\beta}=h^*_{p_i}u*_{i\tau }=0\quad \text{on}\  \partial\Omega^*.
\end{equation}

\noindent $\bullet$ \textbf{Estimate for $u^*_{\beta\beta}$:}

Denote $H^* = h^*(Du^*)$. Note that $H^*$ is positive in $\Omega^*$ and zero on $\partial\Omega^*$. We compute its derivatives:
\[D_i H^* = h_{p_k}^* D_{ik} u^*\]
and
\[
D_{ij} H^* = h^*_{p_k} D_{ijk} u^* + h_{p_kp_l}^* u^*_{ik} u^*_{jl}.
\]
Recall the linear operator $\mathcal{L^*}$ for dual equation:
\[ \mathcal{L}^* f = - \dot{f} + w^* (F^*)^{-2} (F^*)^{ij}w^* b^*_{ik} (D_{kl} f) b^*_{lj},
\]for any function $f$.

Differentiating the equation $	(u^*)_t = -w^* (F^*)^{-1}$ with respect to \(y_k\), we obtain
\begin{align*}
	\dot{u}^*_{k} &= -\frac{y_k}{w^* }(F^*)^{-1} + w^* (F^*)^{-2} F^{*ij}\frac{y_k}{w^* }b^*_{ip}D_{pq}u^*b_{qj}^*\\
&	+ 2 w^* (F^*)^{-2} w^*\frac{\partial b_{ip}^*}{\partial y_k} D_{pq} u^* b_{qj}^*F^{*ij}+ w^* (F^*)^{-2}F^{*ij} w^* b^*_{ip} D_{pqk} u^* b_{gj}^*.
\end{align*}
Then we have
\[\mathcal{L}^*u_k^*=\frac{y_k}{w^* }(F^*)^{-1} - w^* (F^*)^{-2} F^{*ij}\frac{y_k}{w^* }b^*_{ip}u_{pq}^*b_{qj}^* - 2 w^* (F^*)^{-2}F^{*ij} w^*\frac{\partial b_{ir}^*}{\partial y_k} ( b^*)^{rs} b^*_{s p} u_{pq}^* b_{qj}^*\]
Noting that
\[
\frac{\partial b^*_{ir}}{\partial y_k} = \frac{y_i \delta_{k r} + y_{r} \delta_{i k}}{1 + w^*} - \frac{y_i y_r}{(1 + w^*)^2} \frac{y_k}{w^*}, \quad (b^*)^{ rs} = \delta_{rs} - \frac{y_{r} y_s}{w^* (1 + w^*)}
\]
and $|\sum_{p,q}F^{*ij} w^*  b^*_{s p}u^*_{pq} b_{qj}^*|$ is uniformly bound, we derive that
\begin{equation}
	\left| \mathcal{L}^* u_k^* \right| \leq C,
\end{equation}
by the convexity of $u^*$.

Direct computations yield
\[
\mathcal{L}^* H^* = - \dot{H}^* + w^* (F^*)^{-2} F^{*ij} w^* b^*_{ik} (D_{kl} H^*) b_{lj}^*
\]
and
\[
\dot{H}^* = h_{p_k}\dot{ u}^*_k.
\]
Thus
\[
\mathcal{L}^* H^* = h_{p_k} \mathcal{L}^* u_k^* + \sum h^*_{p_kp_l} F^{*ij} w^* b^*_{ip} u^*_{pk}u^*_{ql} b_{qj}^*.
\]
Utilizing the following identity
\[
\sum_{p,q}F^{*ij} w^* b_{i p}^* u^*_{ pk} u^*_{q l}b_{qj}^*=\sum_{p,q,s,n}F^{*ij} w^* b_{i p}^* u^*_{p s} b_{s t}^* (b^*)^{tk} (b^*)^{lm} b^*_{mn} u^*_{ nq} b_{ qj}^*,
\]
 we can derive
\begin{equation}\label{ukk}
	\left| \sum_{k,l,p,q}h^*_{p_k p_l} F^{*ij} w^* b_{i p}^* u^*_{ pk} u^*_{q l} b_{q j}^* \right| \leq C \sum_i F^{*ii} \lambda_i^2,
\end{equation}
where $\lambda_1,\cdots,\lambda_n$  are eigenvalues of the matrix $(\sum_{k,l}w^* b_{ik}^* u^*_{k l} b_{lj}^*)$, which also corresponds the principle radii of curvature of $M_u$. Thus we obtain
\[
\left| \mathcal{L}^* H^* \right| \leq C \left( 1 + \sum_i F^{*ii} \lambda_i^2 \right).
\]
Since $F^*=(\frac{\sigma_n}{\sigma_{n-k}})^{\frac{1}{k}}$, by \cite{Urbas2001}, for any $\epsilon>0$, we further get
\[
\left| \mathcal{L}^* H^* \right| \leq C \left( 1 + (C(\epsilon) + \epsilon M) \sum_i  F^{*ii} \right),\]
where $C(\epsilon)$ is a constant depending only on $\epsilon$ and
\begin{equation}\label{M}
	M = \sup_{y\in\Omega^*}\{\lambda_1(y),\cdots,\lambda_n(y)\}.
\end{equation}

Since \eqref{dualeq} is invariant under rotations in $\mathbb{R}^n$, we may assume the positive $y_{n}$-axis is the interior normal direction to $\partial\Omega^*$ at $ \hat{y} =( \hat{y} _1, \hat{y} _2,\cdots, \hat{y}_n)$.
Near $y_0$, the boundary $\partial\Omega^*$ is given by
\begin{equation*}
	\bar{y}_n=y_{n}-\hat{y} _n=\rho^*(y')=\frac{1}{2}\sum_{\alpha<n}\kappa^{*}_{\alpha}(y_{\alpha}- \hat{y} _{\alpha})^2+O(|y'|^{3}),
\end{equation*}
where $\kappa^{*}_{1},\kappa^{*}_{2},\cdots,\kappa^{*}_{n-1}$ are the principal curvatures of $\partial\Omega^*$  at $\hat{y}$ and $y'=(y_{1}-\hat{y}_1,y_{2}-\hat{y}_2,\cdots,y_{n-1}-\hat{y}_{n-1})$.
Let $	\Omega^*_{r}$ be a neighborhood of $y_{0}$ in $\Omega^*$, defined as
\begin{equation*}
	\Omega^*_{r}:=\{y\in\Omega^*:\rho^*(y')< \bar{y}_{n}<\rho^*(y')+r^{2}, |y'|<r\},
\end{equation*}
where $r$ is a small positive constant to be chosen. We now construct a barrier function $\varphi^*$ defined as follows:
\begin{equation}\label{boundary*}
	\varphi^*(y)=-\rho^*(y')+ \bar{y}_{n}+\delta^*|y'|^{2}-K^*\bar{y}_{n}^2,
\end{equation}
where $\delta^*=\frac{1}{6}\min\{\kappa^{*}_{1},\kappa^{*}_{2},\cdots,\kappa^{*}_{n-1}\}$ and $K^*$ is some large, undetermined positive constant. The function $\varphi^*$ is designed to satisfy the following properties:
\begin{equation*}
	\mathcal{L^*}\varphi^*
	\leq -C\sum_iF^{*ii}\quad\text{in}\ \Omega^*_{r} \ \ \text{ and }\quad \varphi^*\geq 0 \text{ on } \p\Omega_r^*.
\end{equation*}
For any boundary point $\hat{y}\in\p\Omega^*$, we can define a neighborhood  $\Omega^*_r$ of $\hat{y}$ and barrier function $\varphi^*$.  Next, we construct the following test function:
\[
\Phi = H^* - \frac{B}{2} (C(\epsilon) + \epsilon M) \varphi^* \quad \text{in }\Omega_r^*,
\]
where $B$) is a positive constant.

At $t=0,$
\[
\Phi(y,0) = H^*(y,0) -\frac{B}{2} (C(\epsilon) + \epsilon M) \varphi^*(y) \leq 0.\]
On the boundary \(\partial \Omega^*\), it holds that:
\[\Phi\leq 0.\] Moreover, we compute:
\begin{align*}
	\mathcal{L}^*\Phi &= \mathcal{L}^* H^* - \frac{B}{2} (C(\epsilon) + \epsilon M) \mathcal{L}^* \varphi^*\\
	&\geq -C (C(\epsilon) + \epsilon M) \sum F^{*ii} + B (C(\epsilon) + \epsilon M) \sum F^{*ii} - C \sum F^{*ii}\\
	& \geq 0.
	\end{align*}
If $r$ is  chosen sufficiently small, we have $\Phi\leq 0$ on $\p\Omega_r^*$ and $\Phi(y_0)=0$. It follows that $\Phi$ achieves its maximum value at $\hat{y}$. Consequently, at $\hat{y}$,  we drive
\[
D_{\nu^*} H^* \leq C \left( C(\epsilon) + \epsilon M \right).
\]
Since $\hat{y}$ is arbitrary, combining with \eqref{estimatetau}, we conclude
\begin{equation}\label{nn}
	0 \leq D_{\beta \beta} u^* \leq C(\epsilon) + \epsilon M \quad \text{on} \quad \partial \Omega^*.
\end{equation}

\noindent $\bullet$ \textbf{Estimate of $u^*_{\tau \tau }$:}

In estimating the double tangential derivative $u^*_{\tau \tau }$, we encounter the additional negative terms that cannot be eliminated directly. To address this, we adopt a strategy similar to that in \cite{HQWW} for the elliptic curvature equation. By using the orthogonal invariance of the hypersurfaces, we are able to differentiate along specific tangential vector fields, which allows us to establish the estimate for $u^*_{\tau \tau }$.

Let $\{A_s\}$ be a family of orthogonal matrices in $\mathbb{R}^{n+1}$, where $s\in[0,\epsilon_0]$ and $\epsilon_0$ is a small positive constant. Define
\begin{eqnarray}\label{sigmat}
	\sigma_s(y)=PA_s^{-1}P^{-1}(y),\quad y\in\mathbb{R}^n,
\end{eqnarray}
where $P: \mathbb{S}^n_+\rightarrow \mathbb{R}^n$ is the projection map  and $P^{-1}, A_s^{-1}$ are inverse maps of $P, A_s$ respectively.

Following Lemma 5.2 of \cite{HQWW}, we can prove

\begin{lemma}
	Let
	\[
	u_s^*(y, t) = u^*(\sigma_s(y), t)
	\]
	and
	\begin{equation}\label{us}
			\tilde{u}_s^*(y, t) = \frac{\sqrt{1+ |y|^2}}{\sqrt{1 + |\sigma_s(y)|^2}} u_s^*(y, t).	
	\end{equation}	
	Then, the evolution equation for $ u_s^*(y, t)$ is
\begin{equation}\label{useq}
		\frac{d}{dt} u_s^*(y, t) = -\sqrt{1 + |\sigma_s(y)|^2} \left( F^* \big( w^* b_{ik}^*(\tilde{u}_s^*)_{kl} b_{lj}^* \big) \right)^{-1}.
\end{equation}

	\end{lemma}
\begin{proof}
	 We set the hypersurface $X$ as a vector depending on $x\in \mathbb{S}^n_+$, where $x$ be the unit normal vector of hypersurface $X$.  The support function $v(x)$, defined as $$v(x,t) = -\langle X(x,t), x \rangle,$$ satisfies the evolution equation:
	  \[
	 \frac{d}{dt} (v(x, t)) = -\big( F^* \big( \Lambda (x, t) \big) \big)^{-1}.
	 \]
	 	
	  Let $X_s = A_s(X)$. The support function $v_s$ of $X_s $ is 	
$$v_s(x) = -\langle X_s(x), x \rangle= v (A_s^{-1} (x), t).$$
Then we have
\begin{align}\label{vseq}
	\frac{d}{dt} (v_s (x, t)) &= \frac{d}{dt} v (A_s^{-1} (x, t)) = -\left( F^* (v_{ij} + v \delta_{ij}) (A_s^{-1} (x), t) \right)^{-1}\\ \nonumber
	&	= -\left( F^* \left( (v_s)_{ij} + v_s \delta_{ij} \right) (x, t) \right)^{-1}.
\end{align}	

By \eqref{us} and $u^*=w^*v$, we get
 \[	\tilde{u}^*_s(y,t)=\sqrt{1+|y|^2} \, v_s \left(P^{-1} y,t\right).
\]
Applying Lemma 2.1, we have
	\[
\sum_{k,l}w^* \, b_{ik}^* (\tilde{u}^*_s)_{kl} \, b_{lj}^* = (v_s)_{ij} + v_s \delta_{ij}.
\]
Therefore, we can obtain the evolution equation \eqref{useq} for $\tilde{u}_s^*(y, t)$ from \eqref{vseq}.

\end{proof}

Next, we take derivatives with respect to the following tangential vector
\[
\mathcal{T} = \sum_{m=1}^n \frac{\partial g_m}{\partial s} \Bigl|_{s=0} \frac{\partial }{\partial y_m},
\]
which is generated by  \(\sigma_s(y) = (g_1(y), \ldots, g_n(y))\).

From Proposition 5.1 in \cite{HQWW}, we have	
\begin{align}\label{newut}
	\left.\frac{d}{ds} \tilde{u}_s^*\right|_{s=0} &=w^*\left.\frac{d}{ds} \frac{u_s^*}{\sqrt{1+|\sigma_s(y)|^2}}\right|_{s=0}=w^*\mathcal{T} \frac{u^*}{w^*},\\	
	\left.\frac{d^2}{ds^2} \tilde{u}_s^*\right|_{s=0} &=w^*\left.\frac{d^2}{ds^2} \frac{u_s^*}{\sqrt{1+|\sigma_t(y)|^2}}\right|_{s=0}=w^*\mathcal{T} ^2\frac{u^*}{w^*}.\nonumber
\end{align}

Taking derivative in \eqref{useq} with respect to $s$ and letting $s=0$, by \eqref{newut}, we obtain
\begin{equation}\label{Td1}
	\frac{d}{dt} \left( w^* \mathcal{T} \frac{u^*}{w^*} \right) = w^* (F^*)^{-2} F^{*ij} w^* b^*_{ik} \left( w^*\mathcal{T}( \frac{u^*}{w^*} )\right)_{kl} b_{lj}^*
\end{equation}
and
\begin{align}\label{Td2}\nonumber
	\frac{d}{dt} \left( w^* \mathcal{T}^2 \frac{u^*}{w^*} \right)& = w^* (F^*)^{-2} F^{*ij} w^* b^*_{ik}
	\left( w^* \mathcal{T}^2  \frac{u^*}{w^*} \right)_{kl} b^*_{lj}- 2 w^* (F^*)^{-3}\left(F^{*ij} w^* b^*_{ik}\left( w^* \mathcal{T} \frac{u^*}{w^*} \right)_{kl} b_{lj}^*\right)^2\\ \nonumber
	&+ w^* (F^*)^{-2} F^{*ij, pq}\left( w^* b^*_{ik} \left(w^*\mathcal{T}  \frac{u^*}{w} \right)_{kl} b_{lj}^*\right)
	\left( w^* b^*_{pr} \left(w^* \mathcal{T} \frac{u^*}{w^*} \right)_{rs} b^*_{sq}\right)\\
	&\leq w^* (F^*)^{-2} F^{*ij} w^* b^*_{ik} (w^* \mathcal{T} \frac{u^*}{w^*})_{kl} b_{lj}^*
	- 2 w^* (F^*)^{-3}\left (F^{*ij} w^* b^*_{ik} \left( w^*\mathcal{T}\frac{u^*}{w^*} \right)_{kl} b_{lj}^*\right)^2.
	\end{align}

 By Proposition 5.2 of \cite{HQWW}, we can construct a vector field \(\mathcal{T}=\sum_m \mathcal{T}_m \frac{\partial}{\parital y_m} \) near any given boundary point \(y_0 \in \partial \Omega^*\), which parallel any given tangential vector at $y_0$, and such that
\[
\mathcal{T} _m(y_0) = \sqrt{1 + |y_0|^2} \xi_m
\]
and
\[
|\mathcal{T} (y)|^2 \leq 1 + |y|^2,
\]
where the equality holds for \( y = y_0 \). Moreover, we have \eqref{Td1} and \eqref{Td2}.

The remaining proof follows Step 3 (double tangential estimate) in \cite{HQWW} directly. We here include the detailed steps for the completeness of the proof.

Let
\begin{equation}\label{M}
	\tilde{M} = \max_{(y,t) \in \partial \Omega^* \times [0,T] \atop \eta \in T_y(\partial\Omega^*), |\eta|=1} (1 + |y|^2) D_{\eta \eta} u^*(y, t).
\end{equation}

Assume the maximum value $\tilde{M}$ is achieved at \((y_0, t_0)\) and along the unit tangential  direction $\xi \in T_{y_0}(\partial \Omega^*)$, i.e.
\[
\tilde{M}=(1+|y_0|^2) D_{\xi\xi}u^*(y_0, t_0).
\]
Let $x_0=P^{-1}(y_0)$ and $e_1$ be the unit vector which has the same direction as $dP^{-1}(\xi)$.
Then we construct a vector field $\mathcal{T}$ around $y_0$ with $\mathcal{T}(y_0)=\sqrt{1+|y_0|^2}\xi$. We rewrite $\mathcal{T}$ as
\[
\mathcal{T} = \tau(\mathcal{T}) + \frac{\langle\nu^*,\mathcal{T}\rangle}{\langle\beta,\nu^*\rangle} \beta \quad \text{on} \quad \partial \Omega^*,
\]
where $\tau(\mathcal{T})$ is the tangential component of $\mathcal{T}$ on the tangential space of $\partial \Omega^*$ and $\nu^*$ is the unit interior normal of $\p\Omega^*$.
Let
\[
\beta = \beta^t + \langle\beta,\nu^*\rangle\nu^*,
\]
where $\beta^t$ is the tangential component of $\beta$ on the tangential space of $\p\Omega^*$.

By $u^*_{\tau\beta}=0$, we get
\begin{equation}\label{B1}
	D_{\mathcal{T}\mathcal{T}} u^* = D_{\tau(\mathcal{T}) \tau(\mathcal{T})} u^* + \frac{\langle\nu^*,\mathcal{T}\rangle^2}{\langle\beta,\nu^*\rangle^2} D_{\beta \beta} u^*.
\end{equation}
Moreover, we have
\[
|\tau(\mathcal{T})|^2 \leq |\mathcal{T}|^2 + C \langle\nu^*, \mathcal{T}\rangle^2 - 2 \langle\nu^*, \mathcal{T}\rangle \frac{\langle\beta^t, \mathcal{T}\rangle}{\langle\beta, \nu^*\rangle}.
\]
i.e.
\[
|\tau(\mathcal{T})|^2 \leq  \left(1+|y|^2\right) \left( |\tilde{\mathcal{T}}|^2 + C \langle\nu^*, \tilde{\mathcal{T}}\rangle^2 - 2 \langle\nu^*, \tilde{\mathcal{T}}\rangle \frac{\langle\beta^t, \tilde{\mathcal{T}}\rangle}{\langle\beta, \nu^*\rangle} \right)
\]
for $\tilde{\mathcal{T}} = \frac{1}{\sqrt{1+|y|^2}}\mathcal{T}$. Setting  $\eta=\tau(\mathcal{T})/|\tau(\mathcal{T})|$ be the unit tangential vector field, we derive
\[
D_{\tau(\mathcal{T})\tau(\mathcal{T})} u^*(y) \leq  \left( |\tilde{\mathcal{T}}|^2 + C \langle\nu^*, \tilde{T}\rangle^2 - 2 \langle\nu^*, \tilde{\mathcal{T}}\rangle \frac{\langle\beta^t, \tilde{\mathcal{T}}\rangle}{\langle\beta, \nu^*\rangle} \right) \left( 1 + |y|^2 \right)  D_{\eta \eta} u^*(y).
\]
Therefore,  using the positive lower bound estimate of $\langle\beta,\nu^*\rangle$, the estimate for $u^*_{\beta\beta}$, and definition of $\tilde{M}$, we obtain
\begin{eqnarray}\label{B2}
\begin{aligned}
	\frac{D_{\mathcal{T}\mathcal{T}} u^*}{\tilde{M}}\leq   |\tilde{\mathcal{T}}|^2&  + C \langle\nu^*, \tilde{\mathcal{T}}\rangle^2 - 2 \langle\nu^*, \tilde{\mathcal{T}}\rangle \frac{\langle\beta^t, \tilde{T}\rangle}{\langle\beta, \nu^*\rangle}\\
 &\qquad\qquad+ \frac{C}{\tilde{M}}(C(\epsilon) +\epsilon M)\langle\nu^*, \tilde{\mathcal{T}}\rangle^2 \quad \text{on}\ \partial \Omega^*.
\end{aligned}
\end{eqnarray}

Define an auxiliary function:
\[
\phi= \frac{D_{\mathcal{T}\mathcal{T}} u^*}{\tilde{M}} + 2 \left< \nu^*, \tilde{T} \right> \frac{\left< \beta^T, \tilde{\mathcal{T}} \right>}{\left< \beta, \nu^* \right>} - A H^*,
\]
where $A$ is a positive constant to de determined.
 At the point $(y_0, t_0)$, since $\mathcal{T}(y_0)=\sqrt{1+|y_0|^2}\xi$ and$H^*=0$ on $\p\Omega^*$, we have
\[
\tilde{\mathcal{T}} = \xi,    \quad \text{and} \quad D_{\mathcal{T}\mathcal{T}} u^* = \tilde{M}
\]
and
\begin{equation}\label{B21}
	\phi(y_0, t_0) = 1.
\end{equation}
Combing \eqref{supest1} \eqref{estimatetau} and the definitions of $M$, $\tilde{M}$, we have
\begin{equation}\label{MM}
	M \leq C \left( C(\epsilon)+\epsilon M + \tilde{M} \right).
\end{equation}
By the same computation in \cite{HQWW}, we can derive the following estimate near $y_0$:
\begin{equation}\label{ww}
\phi \leq 1 + C_1(\epsilon) \left| y - y_0 \right|^2
\end{equation}
for sufficiently small $\epsilon$, where $C_1(\epsilon)$ is a constant  depending only on $\epsilon$.

Next, define a test function
\begin{equation}
	\Psi =\phi - C_1(\epsilon) |y - y_0|^2 - B(C(\epsilon) + \epsilon M) \varphi^*,
	\end{equation}
where $B$ is a positive constant to be determined and $\varphi^*$ is defined in \eqref{boundary*}.

Since $\varphi^*\geq 0$ on $\p\Omega^*_r\cap \p\Omega^*$ for $r$ sufficiently small, using \eqref{ww}, $\Psi$ achieves its maximum value at $(y_0,t_0)$ on $\p\Omega^*_r\cap \p\Omega^*$. Therefore, if $B$ is chosen sufficiently large, we get
\begin{equation}\label{Psi1}
	\Psi\leq 1\quad\text{on}\quad\p\Omega_r^*\cap\Omega^*.
\end{equation}
We also have
\begin{equation}
	 \triangle \Psi(y, 0) \geq 0, \quad \text{in} \quad \Omega^*_r \times \{0\}.
 \end{equation}
On the boundary $y\in \p\Omega^*_r\cap \Omega^*$,  if $y$ is close to $y_0$, we use \eqref{ww} to directly obtain
\begin{equation}\label{Psi3}
		\Psi \leq 1,\quad y\in \p\Omega^*_r\cap \Omega^*.
\end{equation}
And $y$ is away from $y_0$, we choose $C_1(\epsilon)$ sufficiently large to ensure \eqref{Psi3} holds.
Combing \eqref{Psi3} with \eqref{Psi1}, we derive
\begin{equation}\label{Psi2}
	\Psi \leq 1 \quad \text{on} \quad \partial\Omega^*_r \times \{0\}.
\end{equation}
Then, by the maximum principle, we derive
\begin{equation}\label{Psi2}
	\Psi \leq 1 \quad \text{in} \quad \Omega^*_r \times \{0\}.
\end{equation}
We now proceed to calculate and obtain the following estimate:
\[
\mathcal{L}^* (|y - y_0|^2) \leq C \sum_i F^{*ii},\quad \mathcal{L}^* \varphi^* \geq C \sum F^{*ii}.
\]

Let $\zeta(y,p)$ be a smooth function depending on $y$ and $p=Du^*$:
$$\zeta= \zeta (y, p) = 2 \langle\nu^*, \tilde{T}\rangle \frac{\langle\beta^t, \tilde{\mathcal{T}}\rangle}{\langle\beta, \nu^*\rangle}.$$
Then by the definition of $\phi$, we get
\[
\mathcal{L}^* \phi = \frac{1}{\tilde{M}} \mathcal{L}^* D_{\mathcal{T}\mathcal{T}} u^* + \mathcal{L}^*\zeta - A \mathcal{L}^* H^*.
\]
Thus, we get
\begin{align*}
	\mathcal{L}^*( D_{\mathcal{T}\mathcal{T}} u^*) &= \mathcal{L}^* \left( w^* \mathcal{T}^2( \frac{u^*}{w^*} )+ 2 T u^* \frac{\mathcal{T} w^*}{w^*} + u^*  \frac{\mathcal{T}^2w^*}{w^*} -2 u^* \frac{(\mathcal{T}w^*)^2}{(w^*)^2} - D_{\mathcal{T}}\mathcal{T}(u^*)\right).\\
	&\geq \mathcal{L}^* \left(2 \mathcal{T} u^* \frac{\mathcal{T}w^*}{w^*} + u^*  \frac{\mathcal{T}^2w^*}{w^*} - 2 u^* \frac{(\mathcal{T}w^*)^2}{(w^*)^2} - D_{\mathcal{T}}\mathcal{T}(u^*)\right)\\
	&\geq \mathcal{L}^* \left(2 \mathcal{T} u^* \frac{\mathcal{T}w^*}{w^*}  - D_{\mathcal{T}}\mathcal{T}(u^*)\right) - C \sum_i F^{*ii}-C\\
&\geq -C- C \sum_i F^{*ii}.
\end{align*}
Direct calculation yields
\begin{align*}
	\dot{\zeta}&=\sum_i\zeta_{p_i}\dot{u}_i^*,\quad \zeta_{i} = \zeta_{y_i} + \sum_s\zeta_{p_s} u^*_{si},\\
	\zeta_{ij} &= \zeta_{y_iy_j} +\sum_s \zeta_{y_i p_s} u^*_{sj} + \sum_s\zeta_{p_s y_j}u^*_{si}+ \sum_{s,t}\zeta_{p_sp_t} u^*_{si}u^*_{tj}+\sum_s\zeta_{p_s}u^*_{sij},
\end{align*}
Then we have
\begin{align*}
	\mathcal{L}^* \zeta& = -\dot{\zeta} +  w^* (F^*)^{-2} F^{*ij} w^* b^*_{ip} D_{pq}\zeta b^*_{qj}\\
	& = \sum_{i} \zeta_{p_i}\mathcal{L}^* \dot{u}_i^*+  w^* (F^*)^{-2} F^{*ij} w^* b^*_{ip} D_{pq}\zeta b^*_{qj}\\
	&\geq -C \sum_i F^{*ii} -C-C\sum F^{*ij}w^* b_{ik}  u_{sk}^* u^*_{sl}b^*_{lj}.
\end{align*}
Thus, assuming that $A$ and $B$ are sufficiently large, we can obtain
\[
\mathcal{L}^* \Psi \geq 0.
\]
From this we deduce that $\Psi$ achieves its maximum at the point $y_0$ at time $t_0$. Consequently, we have the estimate:
	\[
D_\beta \phi(y_0) \leq C(C(\epsilon) + \epsilon M).
\]
This then implies
\[
D_{\mathcal{T}\mathcal{T}\beta} u^*(y_0) \leq C(C(\epsilon) + \epsilon M) \tilde{M},
\]
which leads to
\begin{equation}\label{LL1}
	D_{\xi\xi \beta} u^*(y_0) \leq C(C(\epsilon) + \epsilon M) \tilde{M},
\end{equation}
by $\mathcal{T}(y_0)=\sqrt{1+|y_0|^2}\xi$.

Now, by applying the boundary condition $h^*(Du^*) = 0$, at $y_0$, we differentiate twice with respect to $\xi$. This gives us the following equation:
\begin{equation}\label{LL2}
	D_{\xi\xi \beta} u^* +  \sum_{k,l}h^*_{p_kp_l} D_{\xi k} u^* D_{\xi l} u^* +II(\xi,\xi)D_{\nu^* \beta}u^*=0,\end{equation}
where $II$ denotes the second fundamental form of the boundary $\p\Omega^*$ with respect to the unit normal vector $\nu^*$.
Next, combining \eqref{LL1} with \eqref{LL2}, we obtain
\[ - h^*_{p_kp_l} D_{\xi k} u^* D_{\xi l} u^* \leq C(C(\epsilon) + \epsilon M) \tilde{M}.
\]
By utilizing the concavity of $h^*$, we derive the following bound on the second derivative of $u^*$:
\[
(D_{\xi\xi} u^*)^2 \leq C \left( C(\epsilon) + \epsilon M \right) \tilde{M}.
\]
From the definition of  $\tilde{M}$ we obtain
\[
\tilde{M} \leq C \left( C(\epsilon) + \epsilon M \right).
\]
Finally, combining this with \eqref{MM} and choosing $\epsilon$ sufficiently small, we conclude that
\[
M\leq C(\epsilon) .
\]
Thus, we conclude that the double tangential derivative $u^*_{\tau\tau}$ is uniformly bounded.

Combing the global $C^2$ estimate and boundary  $C^2$ estimates, we obtain the following $C^2$ estimate:
\begin{lemma}
	For a solution of flow \eqref{dualeq}, the following a priori bound holds for the second derivatives of $u^*$:
	\[
	\sup_{\bar{\Omega^*}\times[0,T]} D^2 u^* \leq C.
	\]
\end{lemma}
Since $F^*=\left( \frac{\sigma_n}{\sigma_{n-k}} \right)^{\frac{1}{k}}=\left[\frac{\sigma_n}{\sigma_{n-k}}\Big(\la^*\Big[w^*\sum_{k,l} b^*_{ik}(D_{kl}u^*)b^*_{lj}\Big]\Big)\right]^{\frac{1}{k}}$ was proved bounded in Section 5 and by Lemma 7.2, we can deduce a positive lower bound for $F^*$. This leads to the estimate for $F= \sigma_k^{\frac{1}{k}}$, from which we can derive bounds for the principle curvatures. Thus, we obtain the following estimate for $u$:
\begin{lemma}
	For a solution of flow \eqref{uflow}, the following a priori estimate holds for $u$:
	\[
	\sup_{\bar{\Omega}\times[0,T]} D^2 u \leq C.
	\]
\end{lemma}

\section{Proof of Theorem 1.1}
In this section, we present the proof of Theorem 1.1. We begin by establishing the long-term existence of the flow. Using a standard argument based on the implicit function theorem, we conclude that \eqref{uflow} admits a unique solution for short time. In the previous sections, we have established $C^0$ and uniform $C^2$ estimates for the solution $u$. This enables us to obtain a uniform $C^{2,\alpha}$ estimate. Applying standard parabolic theory, we next establish uniform bounds for higher-order derivatives, thereby ensuring the existence of a smooth solution to \eqref{uflow} for all $t \geq 0$.

Next, we show that the solution $u$ smoothly converges to a translating solution.

Since we have established the existence of the flow and obtained the uniform bounds for the solution, we can apply Theorem 1.1 from \cite{Huang2} for the second boundary value problem of nonlinear parabolic equations to derive the convergence. This convergence theorem guarantees that the solution converges to a solution the corresponding elliptic equation, which describes the translating solution,  provided the solution exists for a sufficiently long time and the uniform boundary estimates hold.  We have derived the existence of a translating solution for \eqref{uflow} in Section 3. The uniform bounds ensure that the solution remains bounded and satisfies the assumptions required by the convergence theorem. Therefore, applying this result, we conclude that the flow \eqref{uflow}  converges to a translating solution. In addition to directly using the convergence theorem from \cite{Huang2}, we provide an independent proof, employing an argument similar to the one in \cite{Schnurer2002}.

 A translating solution takes the form $u(x,t)=u(x)+at$. We denote the translating solution derived in Section 3 by $u^\infty(x,t)$. Define
\[W:=u(x,t)-u^\infty(x,t).\]
By the mean value theorem, there exist a positive definite matrix $a_{ij}$ and a vector field $b_i$ such that $W$ satisfies linear parabolic equation:
\[\dot{W} =a^{ij}W_{ij}+b^iW_i,\quad\text{in } \Omega \times [0, \infty).\]

This completes the proof that any solution of flow \eqref{uflow} exists for all times and converges smoothly to a translating solution.

On the boundary, we have
\begin{align*}
	0&=h(Du(x,t))-h(Du^\infty(x,t))\\
	&=\int_0^1 h_{p_k}(\tau Du+(1-\tau Du^\infty))d\tau \cdot W_k=\beta^k W_k.
\end{align*}
Thus, we obtain the boundary value problem for $W$
\begin{equation}\label{Weq}
	\left\{
	\begin{aligned}
		&\dot{W} =a^{ij}W_{ij}+b^iW_i,  &\text{in } \Omega \times [0, \infty)\\
		&	\beta_kW_k=0, &\text{on } \partial \Omega
	\end{aligned}
	\right.
\end{equation}
By the strong maximum principle, the oscillation of $W$ tends to zero as $t\rightarrow\infty$, implying that $W$ converges to a constant. Consequently, $u$ converges to a translating solution $u^\infty(x,t)$ up a translating as $t \to \infty$. Using interpolation inequalities, we can then establish the smooth convergence of $u$ to the translating solution.



\begin{thebibliography}{DU}
  
\bibitem{Alschuler1994}S. Altschuler, L. Wu, {\it Translating surfaces of the non-parametric mean curvature flow with prescribed contact angle}, Calc. Var. Partial Differential Equations {\bf 2} (1994), 101-111.

   \bibitem{Angenent2019} S. Angenent, P. Daskalopoulos and N. Se\v sum, {\it Unique asymptotics of ancient convex mean curvature flow solutions}, J. Differential Geom., {bf 111}(2019): 381-455.
 
 \bibitem{Angenent2020} S. Angenent, P. Daskalopoulos and N. Se\v sum, {\it Uniqueness of two-convex closed ancient solutions to the mean curvature flow}, Ann. of Math. (2) {\bf 192} (2020), 353-436.   
 
\bibitem{Andrews2020} B. Andrews et al., {\it Extrinsic geometric flows}, Graduate Studies in Mathematics, 206, Amer. Math. Soc., Providence, RI, 2020.

\bibitem{Andrews1999}B. Andrews, Gauss curvature flow: the fate of the rolling stones, Invent. Math. {\bf 138} (1999), 151-161.

\bibitem{Andrews2016}B. Andrews, P. Guan and L. Ni, Flow by powers of the Gauss curvature, Adv. Math. {\bf 299} (2016), 174-201.

\bibitem{Brendle2010} S. Brendle and M. Warren, {\it A boundary value problem for minimal Lagrangian graphs}, J. Differential Geom., {\bf 84} (2010), 267-287.

\bibitem{Brendle2017} S. Brendle, K. Choi and P. Daskalopoulos, {\it  Asymptotic behavior of flows by powers of the Gaussian curvature}, Acta Math. {\bf 219} (2017), 1-16.

\bibitem{Caffarelli1986} L. Caffarelli, L. Nirenberg, and J. Spruck, {\it Nonlinear second order elliptic equations. IV. Starshaped compact Weingarten hypersurfaces}, in Current topics in partial differential equations, Kinokuniya, Tokyo, 1986, 1-26.

\bibitem{Chen2016} S. Chen and A. Figalli, {\it Stability results on the smoothness of optimal transport maps with general costs}, J. Math. Pures Appl. {\bf 106} (2016), 280-295.


\bibitem{Chen2021} S. Chen, J. Liu, and X.J. Wang, {\it Global regularity for the Monge-Amp\`{e}re equation with natural boundary condition}, Ann. of Math. {\bf 194} (2021), 745-793.
\bibitem{Choi2022} B. Choi, K. Choi and P. Daskalopoulos, {\it Convergence of Gauss curvature flows to translating solitons}, Adv. Math. {\bf 397} (2022), Paper No. 108207, 30 pp.


\bibitem{EckerHuisken1991}K. Ecker and G. Huisken, Interior estimates for hypersurfaces moving by mean curvature, Invent. Math. {\bf 105} (1991), 547-569.


\bibitem{Firey}W. Firey, {\it Shapes of worn stones}, Mathematika {\bf 21} (1974), 1-11.


\bibitem{Huang2015} R. Huang, {\it On the second boundary value problem for Lagrangian mean curvature flow}, J. Funct. Anal. {\bf 269} (2015), 1095-1114.

\bibitem{Huang} R. Huang and Q. Ou, {\it On the second boundary value problem for a class of fully nonlinear equations}, J. Geom. Anal. 27 (2017), 2601C2617.

\bibitem{HQWW} R. Huang, C. Qu, Z. Wang, W. Wo, {\it Hessian curvature hypersurfaces with prescribed Gauss image}. arXiv preprint arXiv:2412.09159, 2024.

\bibitem{Huang2}R. Huang and Y. Ye, A convergence result on the second boundary value problem for parabolic equations, Pacific J. Math. {\bf 310} (2021), 159-179.

\bibitem{Huisken1984}G. Huisken, {\it Flow by mean curvature of convex surfaces into spheres}, J. Differential Geom. {\bf 20} (1984), 237-266.

\bibitem{Huisken1989}G. Huisken, {\it Nonparametric mean curvature evolution with boundary conditions}, J. Differential Equations {\bf 77} (1989), 369-378.



\bibitem{Jiang2018} F. Jiang and N. Trudinger,{\it On the second boundary value problem for Monge-Amp\`{e}re type equations and geometric optics}, Arch. Ration. Mech. Anal. {\bf 229} (2018),  547-567.

\bibitem{Kitagawa2012}J. Kitagawa, {\it A parabolic flow toward solutions of the optimal transportation problem on domains with boundary}, J. Reine Angew. Math. {\bf 672} (2012), 127-160.

\bibitem{Schnurer2002} O. Schn\"urer, {\it Translating solutions to the second boundary value problem for curvature flows}, Manuscripta Math. {\bf 108} (2002), 319-347.

\bibitem{Schnurer20022}O. Schn\"urer, {\it The Dirichlet problem for Weingarten hypersurfaces in Lorentz manifolds}, Math. Z. {\bf 242} (2002), 159--181.

\bibitem{Schnurer2003} O. Schn\"urer and K. Smoczyk, {\it Neumann and second boundary value problems for Hessian and Gauss curvature flows}, Ann. Inst. H. Poincar\'{e} C Anal. Non Lin\'{e}aire {\bf 20} (2003), 1043-1073.

\bibitem{Trudinger}N. Trudinger and X. Wang, {\it On the second boundary value problem for Monge-Amp\`ere type equations and optimal transportation}, Ann. Sc. Norm. Super. Pisa Cl. Sci. (5) {\bf 8} (2009), 143-174.

\bibitem{Tso1985}K. Tso, {\it Deforming a hypersurface by its Gauss-Kronecker curvature}, Comm. Pure Appl. Math. {\bf 38} (1985), 867-882.

\bibitem{Urbas1997} J. Urbas, {\it On the second boundary value problem for equations of Monge-Amp\`{e}re type}, J. Reine Angew. Math. {\bf 487} (1997), 115-124.

\bibitem{Urbas2001} J. Urbas, {\it The second boundary value problem for a class of Hessian equations}, Comm. Partial Differential Equations, {\bf 26} (2001), 859-882.

\bibitem{Urbas2002} J. Urbas, {\it Weingarten hypersurfaces with prescribed gradient image}, Math. Z. {\bf 240} (2002), 53-82.
\bibitem{Wang2023} C. Wang, R. Huang, and J. Bao, {\it On the second boundary value problem for Lagrangian mean curvature equation}, Calc. Var. Partial Differential Equations {\bf 62} (2023), Paper No. 74.

\bibitem{Wang2024} C. Wang, R. Huang, and J. Bao, {\it On the second boundary value problem for a class of fully nonlinear flow III}, J. Evol. Equ. {\bf 24} (2024), Paper No. 52.

\bibitem{WZ1} Z. Wang and L. Xiao, {\it Entire self-expanders for power of $\sigma_k$-curvature flow in Minkowski space}. J. Funct. Anal. 284 (2023), no. 8, Paper No. 109866, 27 pp. 

\bibitem{WZ2} Z. Wang and L. Xiao, {\it Entire convex curvature flow in Minkowski space}. Calc. Var. Partial Differential Equations 62 (2023), no. 9, Paper No. 252, 27 pp.

\bibitem{WZ3} Z. Wang and L. Xiao, {\it Entire $\sigma_k$ curvature flow in Minkowski space}, Preprint.


\end{thebibliography}
\end{document}